\documentclass[12pt,a4paper]{article}
\usepackage[english]{babel}
\usepackage{graphicx,color}
\usepackage{amsmath}
\usepackage{float}
\usepackage{amsthm}
\usepackage{amssymb}
\usepackage{amsfonts}
\usepackage{amsthm}
\usepackage{enumerate}
\usepackage{hyperref}
\usepackage{authblk}
\usepackage[utf8]{inputenc}
\usepackage[top=1.5 in,bottom=1.5 in, left=1 in, right=1 in]{geometry}

\newtheorem{theorem}{Theorem}[section]
\newtheorem{proposition}{Proposition}[section]

\DeclareMathOperator{\inertia}{inertia}

\title{Nucleation and development of multiple cracks in thin composite fibers}

\author[1]{Arnav Gupta\thanks{ag2592@cornell.edu}\textsuperscript{,}}
\author[1,2]{Timothy J. Healey}

\affil[1]{Field of Theoretical and Applied Mechanics, Cornell University, Ithaca, NY 14853, USA}
\affil[2]{Department of Mathematics, Cornell University, Ithaca, NY 14853, USA}

\date{ }

\begin{document}
\maketitle
\begin{abstract}
We study the nucleation and development of crack patterns in thin composite fibers under tension in this work. A fiber comprises an elastic core and an outer layer of a weaker brittle material. In recent tensile experiments on such composites, multiple cracks were observed to develop simultaneously on the outer layer. We propose here a simple one-dimensional model to predict such phenomenon. We idealize the problem as two axially loaded rods coupled by a linear interfacial condition. The latter can be regarded as an adhesive that resists slip between the two materials. One rod is modeled as a brittle material, and the other a linearly elastic material, both undergoing finite deformations.
\end{abstract}

\section{Introduction}
We study the nucleation and development of crack patterns in thin composite fibers under tension in this work. A fiber comprises an elastic core and an outer layer of a weaker brittle material, as illustrated in Figure\ \ref{mainfig}(a). In recent tensile experiments on such composites, multiple cracks were observed to develop simultaneously on the outer layer \cite{NC,NCP}. We propose here a simple one-dimensional model, in the spirit of Ericksen's famous paper \cite{JE}, to predict such phenomenon. We idealize the problem as two axially loaded rods coupled by a linear interfacial condition, as depicted in Figure \ref{mainfig}(b). The latter can be regarded as an adhesive that resists slip between the two materials. One rod is modeled as a brittle material, and the other a linearly elastic material, both undergoing finite deformations.

Following \cite{RH}, we employ the inverse-deformation formulation to analyze the problem under hard loading. As first observed in \cite{RH}, the infinite deformation gradient associated with fracture registers as zero inverse deformation gradient. As a consequence, the description enables more standard methods of analysis. For the brittle rod, we also incorporate a small quadratic interfacial energy in terms of the inverse-strain gradient. As shown in \cite{RH}, that model predicts spontaneous fracture with surface-energy effects at finite loading. The model from \cite{RH} incorporates neither a damage field nor pre-existing cracks. The difference here in this work is that the brittle material is now interacting with an elastic rod. The route to failure in \cite{RH} is a single crack at one or the other of the ends, after which there is no resistance under hard loading. In contrast, the interaction with the elastic core here allows for multiple, simultaneous cracks of the outer layer to develop at criticality. More importantly, the cracked outer layer continues to interact with the inner core beyond that load. Thus, crack development presents a delicate problem here, not arising in \cite{RH}. Also, our model requires a fourth order displacement formulation in the presence of a unilateral constraint. In particular, the phase plane, as employed in \cite{RH}, is not available to us. 

\begin{figure}
\centering
\includegraphics[width=\linewidth]{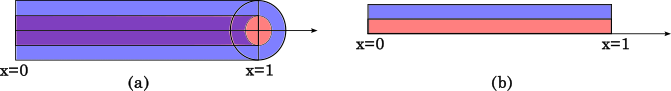}
\caption{(a) Illustrative picture of a coaxial composite fiber. (b) Axisymmetric idealization as a two-bar system.}  
\label{mainfig}
\end{figure}

The outline of the work is as follows. We present our model in Section \ref{model}. We borrow an idea from \cite{VHR}, viz., assume the modulus of the elastic core is much greater than the stiffness of the interfacial adhesion. In this way, the elastic core always behaves homogeneously; the brittle rod can deform homogeneously, while inhomogeneous states are also possible. We henceforth say that the elastic core constitutes a \emph{pseudo-rigid foundation}, in accordance with the terminology of \cite{VHR}. In Section \ref{globalbif} we provide a weak formulation of the problem, enabling a global bifurcation analysis that accounts for the possibility of fractured states along solution branches. We show that the homogeneous solution for the brittle material loses stability at a certain load, where bifurcation to an inhomogeneous family of equilibria arises. As in \cite{RH}, the local solution branch is a subcritical (unstable) pitchfork. However, the first instability here is characterized by a higher-wave solution - much like the buckling of a beam on an elastic foundation \cite{GS}. Motivated by the results from \cite{RH}, the bifurcating solution branch is expected to ``evolve'' to a fractured state - in this case characterized by multiple cracks. However, in the absence of the phase-plane arguments as in \cite{RH}, the question of subsequent fracture or not relies on numerical computation.

We introduce our computational formulation in Section\ \ref{nummeth}, employing a finite element discretization. The implementation is delicate, due to the presence of a unilateral constraint insuring that the inverse deformation gradient is non-negative. We impose the constraint on each node of the mesh and solve for equilibria using the active-set method \cite{NW}. We check the local stability of the computed equilibrium via the second-variation condition. The latter is diagnosed indirectly via the signature of the full Jacobian (of the combined equilibrium and constraint equations) \cite{NW}. We present our computational results in Section \ref{numres} for the same two sets of parameter values used in the local analysis of Section\ \ref{globalbif}. In each case, multiple, simultaneous cracks nucleate on the outer layer on the unstable pitchfork branch, after which  the solutions regain stability as the width of the cracks on the outer layer continue to evolve and widen. Lastly, we locate the critical stretch value beyond which the stable multiple-cracked solutions have lower energy than the homogeneous solution. We end with some concluding remarks in Section\ \ref{conc}.          

\section{The model}\label{model}
\begin{figure}
\centering
\includegraphics[width=0.6\linewidth]{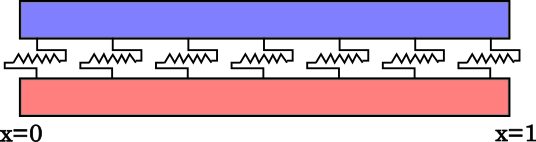}
\caption{The two bars are connected by uniformly distributed springs. This interface condition introduces a resistance to slip between the bars.}  
\label{modelfig}
\end{figure}

Neglecting the radial deformation and assuming an axisymmetric response, we idealize the problem as two axial rods coupled by a linear interface condition, as schematically illustrated in Figure\ \ref{modelfig}. Ignoring strain-gradient effects for the time being, the potential energy of the system is given by
\begin{equation}\label{prob3}
	E[\lambda,f,f_c]=\int_0^1 \left[W(f')+W_c(f_c') +\frac{1}{2}k(f-f_c)^2\right]\,dx,
\end{equation} 
where $f,f_c:\,[0,1]\rightarrow\mathbb{R}$ are deformations of the outer layer and the core, respectively, $W$ and $W_c$ are the respective stored-energy functions, and $k>0$ denotes the stiffness per unit length of the interface. The placement boundary conditions (hard loading) are prescribed as 
\begin{equation}
	f(0)=f_c(0)=0,\qquad f(1)=f_c(1)=\lambda,
\end{equation}
where $\lambda>1$ is the loading parameter. 

For the elastic core, we assume the simple energy function
\begin{equation}
	W_c(f'_c)=\frac{\gamma}{2}(f_c'-1)^2,
\end{equation}
where $\gamma>0$ is the elastic modulus. With this in hand, the first variation of \eqref{prob3} with respect to $f_c$ delivers the Euler-Lagrange equation
\begin{align}\label{ELeq1}
-\gamma f_c''+k(f_c-f)=0.
\end{align} 
Borrowing an idea from \cite{VHR}, we now assume that the stiffness of the core is much higher than the stiffness of the interface springs i.e. $\frac{k}{\gamma}\rightarrow 0$. Equation\ \eqref{ELeq1} then simplifies to
\begin{equation}\label{assump}
	f_c''=0,\quad \implies f_c=\lambda x.
\end{equation} 
Substituting \eqref{assump}\textsubscript{2} into \eqref{prob3}, we find
 \begin{equation}\label{prob5}
	E[\lambda,f]=\int_0^1 \left[W(f')+\frac{1}{2}k(f-\lambda x)^2\right]\,dx+W_c(\lambda),
\end{equation}
 where $W_c(\lambda)$ is the total energy of the core independent of $f$. 
 
Following \cite{RH}, we introduce the inverse mapping $h:\,[0,\lambda]\rightarrow[0,1]$, and include an interfacial energy in \eqref{prob5}, to obtain
 \begin{equation}\label{Estar}
E^*[\lambda,h]=E[\lambda,h^{-1}]=\int_0^\lambda \left[\frac{\epsilon}{2}(h'')^2+W^*(h')+\frac{kh'}{2}(y-\lambda h)^2 \right]\,\,dy +W_c(\lambda),
\end{equation}
\begin{equation}\label{hbc}
	h(0)=0,\quad h(\lambda)=1,\quad h'\geq 0.
\end{equation} 
Here $W^*(H)=HW(1/H)$, where $H=h'\geq 0$ is the inverse deformation gradient and $\epsilon>0$ is a small parameter. Note that $H>0$ on $[0,\lambda]$ implies that the outside material is unbroken, while $H(y)=0$ signifies breakage at $y\in [0,\lambda]$. As employed in \cite{RH}, a prototype for the stored energy function for the outside material is
\begin{equation}\label{lp}
	W(F)=\frac{\beta}{6}\left(1-\frac{1}{F}\right)^2,\qquad \implies W^*(H)=\frac{\beta}{6}H(1-H)^2,
\end{equation}
where $\beta>0$ is the apparent elastic modulus. Observe that $W^*$ is a two-well potential with wells at $H=0$ and $H=1$, while the interfacial term in \eqref{Estar} introduces a transition layer of length $\mathcal{O}(\sqrt{\epsilon})$, cf. \cite{RH}. 

Since $W_c(\lambda)$ is independent of $h$, the problem reduces to solving for minimizers of
 \begin{equation} \label{Istar}
 	I^*[\lambda,h]=E^*[\lambda,h]-W_c(\lambda)=\int_0^\lambda \left[\frac{\epsilon}{2}(h'')^2+W^*(h')+\frac{kh'}{2}(y-\lambda h)^2 \right]\,\,dy
 \end{equation} 
 As in \cite{RH}, we rescale so that the problem is defined on a fixed domain. Let $y=\lambda s,\, s\in[0,1]$ and $u(s)=h(\lambda s)-s$. Then $I^*[\lambda,h]$ scaled by $\lambda^3$ is equivalent to
\begin{equation}\label{uprob}
J^*[\lambda,u]=\lambda^3 I^*[\lambda,h]=\int_0^1\left\{ \frac{\epsilon}{2}(u'')^2+\lambda^4\left[W^*\left(\frac{1+u'}{\lambda}\right)+\frac{k\lambda}{2}u^2(1+u')\right]\right\}\,ds,
\end{equation}
and the geometric boundary conditions \eqref{hbc} become
\begin{equation}\label{ubc}
	u(0)=u(1)=0.
\end{equation}
Moreover, the unilateral constraint now reads $u'\geq -1$ on $[0,1]$.

Finally, we note that the term $u^2u'$, subject to \eqref{ubc}, is a null Lagrangian, i.e., it plays no role in the determination of critical points or equilibria. Hence, we ignore that term in what follows.

\section{Weak formulation}\label{globalbif}
We now provide a weak formulation of the problem that accounts for the possibility of fracture in the brittle layer. First, we define the Hilbert space
\begin{equation}
	\mathcal{H}=H^2(0,1)\cap H_o^1(0,1),
\end{equation}
 equipped with the inner product 
 \begin{equation}\label{norm}
 	\langle u,v\rangle=\int_0^1u''v''\,ds,
 \end{equation}
 the latter of which induces a norm equivalent to the usual one on $H^2(0,1)$. The closed convex cone of admissible states is defined by
 \begin{equation}
 	\mathcal{K}=\{v\in\mathcal{H}:v'\geq-1\},
 \end{equation}
 and we denote the open cone of ``unbroken'' states by
  \begin{equation}\label{Ko}
 	\mathcal{K}^o=\{v\in\mathcal{H}:v'>-1\},
 \end{equation}
 We note that the pointwise characterizations above makes sense due to embedding, viz., $v'$ is continuous on $[0,1]$.
 
 For $u,v\in\mathcal{K}$, and fixed $\lambda\in(0,\infty)$, consider $e_\lambda(t):=J^*[\lambda,(1-t)u+tv]$, for $t\in[0,1]$. Then $u$ is a critical point or an equilibrium (at $\lambda$) if $\dot{e}_{\lambda}\geq 0$, which yields the Euler-Lagrange variational inequality
 \begin{equation}\label{wineq}
 	\int_0^1\left[\epsilon u''(v-u)''+\lambda^3\dot{W}^*\left(\frac{1+u'}{\lambda}\right)(v-u)'+k\lambda^5u(v-u)\right]\,ds\geq 0, \text{ for all } v\in\mathcal{K}.
 \end{equation}
 
 \begin{proposition}\label{P1}
 	An interior point $u\in\mathcal{K}^o$ satisfies \eqref{wineq} (at $\lambda\in(0,\infty)$) if and only if $u\in C^4[0,1]$ and satisfies
 \begin{equation}
 	\begin{aligned}
 		& \epsilon u^{iv}-\lambda^2\ddot{W}\left(\frac{1+u'}{\lambda}\right)u''+k\lambda^5u=0 \text{ in } (0,1),\\
 		&\label{ubc2} u(0)=u(1)=u''(0)=u''(1)=0.
 	\end{aligned}
 \end{equation}
 \end{proposition} 
\begin{proof}
	The boundary conditions \eqref{ubc} are clear. We choose $v=u\pm t\phi$, for $\phi\in\mathcal{H}$, for $\phi\in\mathcal{H}$ and $t>0$ sufficiently small. Then \eqref{wineq} implies
 \begin{equation}\label{weqo}
 	\int_0^1\left[\epsilon u''\phi''+\lambda^3\dot{W}^*\left(\frac{1+u'}{\lambda}\right)\phi'+k\lambda^5\phi\right]\,ds=0, \text{ for all }\phi\in\mathcal{H}.
 \end{equation} 
 By imbedding, we have $u\in C^1[0,1]$, and \eqref{weqo} directly implies that the third distributional derivative of $u$ can be identified with a continuous function. Integrating the first term by parts then yields the other boundary conditions in \eqref{ubc2}\textsubscript{2} and
 \begin{equation}\label{u3sf}
 	\epsilon u'''-\lambda^3\dot{W}^*\left(\frac{1+u'}{\lambda}\right)=-k\lambda^5\int_0^su(\tau)d\tau+C \text{ in }(0,1), 
 \end{equation} 
 where $C$ is a constant. We observe that the left side of \eqref{u3sf} is equal to a continuously differentiable function; \eqref{ubc2}\textsubscript{1} follows upon differentiation of \eqref{u3sf}.
\end{proof} 
 
 For $u\in\mathcal{K}$, we define the broken set
 \begin{equation*}
 	\mathcal{B}_u:=\left\{s\in[0,1]:u'(s)=-1\right\}.
 \end{equation*}
 
 \begin{theorem}
 	Any solution $u\in\mathcal{K}$ of \eqref{wineq} (at $\lambda\in(0,\infty)$) is twice continuously differentiable on $[0,1]$.
 \end{theorem}
 \begin{proof}
 	We let $\psi\in\mathcal{H}$ with $\psi'\geq 0$ on $\mathcal{B}_u$, and set $v=u+\psi\in\mathcal{K}$ in \eqref{wineq} to find
 	 \begin{equation*}
 	\int_0^1\left[\epsilon u''\psi''+\left\{\lambda^3\dot{W}^*\left(\frac{1+u'}{\lambda}\right)-k\lambda^5\left(\int_0^1u(\tau)d\tau\right)\right\}\psi'\right]\,ds\geq 0,
 \end{equation*}
 for all such $\psi$. By the Riesz-Schwarz theorem \cite{KS}, there is a non-negative measure $\mu$, with support in $\mathcal{B}_u$, such that
 \begin{equation}\label{u3ineq}
 	-\epsilon u'''+\lambda^3\dot{W}^*-k\lambda^5\int_0^su(\tau)d\tau-C_1=\mu
 \end{equation}
 in the distributional sense, where $C_1$ is a constant. Integration of \eqref{u3ineq} then yields
 \begin{equation}
 	\epsilon u''=\lambda^3\int_0^s\left\{\dot{W}^*\left(\frac{1+u'(\xi)}{\lambda}\right)-k\lambda^5\int_0^\xi u(\tau)\,d\tau\right\}\,d\xi-C_1s-C_2-\phi(s),
 \end{equation}
 where $C_2$ is another constant and $\phi'=\mu$ in the sense of distributions. Since $u\in C^1[0,1]$ and $\phi$ is non-decreasing, the same argument used in \cite[Theorem~2.3]{RH}, can be employed here to deduce that $u''$ is continuous on $[0,1]$.
 \end{proof}
 
 For the purpose of a global bifurcation analysis that includes the possibility of broken solutions ($\mathcal{B}_u$ is non-empty), we follow the approach in \cite{LS}, also employed in \cite{RH}. First, we express \eqref{wineq} abstractly via
 \begin{equation}\label{opeq}
 	\langle\epsilon u+G(\lambda,u),v-u\rangle\geq 0 \text{ for all } v\in\mathcal{K},
 \end{equation} 
 where $G:(0,\infty)\times\mathcal{H}\rightarrow\mathcal{H}$ is the compact mapping defined by
 \begin{equation}\label{Gdef}
 	\langle G(\lambda,u),\phi\rangle:=\int_0^1\left\{\lambda^3\dot{W}^*\left(\frac{1+u'}{\lambda}\right)-k\lambda^5\int_0^su(\tau)d\tau\right\}\phi'\,ds,
 \end{equation}
for all $\phi\in\mathcal{H},\,\lambda\in(0,\infty)$. 

Let $P_{\mathcal{K}}$ denote the closest-point projection from $\mathcal{H}$ onto $\mathcal{K}$, and define $F:=P_{\mathcal{K}}\circ G$. Then \eqref{opeq} is equivalent to the operator equation
 \begin{equation}\label{Fdef}
 	\epsilon u+F(\lambda,u)=0,
 \end{equation}
 where $F:(0,\infty)\times\mathcal{K}\rightarrow\mathcal{K}$ is continuously differentiable and compact (since $P_{\mathcal{K}}$ is continuous). Thus, the Leray-Schauder degree is well defined for the mapping $u\mapsto\epsilon u+F(\lambda,u)$, for each $\lambda\in(0,\infty)$, e.g., \cite{K}.
 
 Clearly, $0\in\mathcal{K}^o$, and thus, $F(\lambda,0)\equiv P_{\mathcal{K}}(G(\lambda,0))$. From \eqref{Gdef}, we find that
 \begin{equation*}
 	\langle G(\lambda,0),\eta\rangle=0, \text{ for all } \eta\in\mathcal{H},\,\lambda\in(0,\infty) \implies
 \end{equation*}
 $F(\lambda,0)=G(\lambda,0)\equiv 0$, i.e., \eqref{Fdef} admits the trivial line of solutions $u=0$ for all $\lambda\in(0,\infty)$. The rigorous linearization of \eqref{Fdef} about the trivial line is denoted abstractly via
 \begin{equation}\label{Ldef}
 	\epsilon\eta+L(\lambda)\eta=0, \text{ for all } \eta\in\mathcal{H},
 \end{equation}
 where $L(\lambda):=D_uG(\lambda,0)$ denotes the Fr\'echet derivative with respect to $u$. From \eqref{norm} and \eqref{Gdef}, we see that \eqref{Ldef} is equivalent to
 \begin{equation}\label{lwko}
 	\int_0^1\epsilon \eta''\phi''+\lambda^2\left\{\ddot{W}^*\left(\frac{1}{\lambda}\right)\eta'-k\lambda^3\int_0^s\eta(\tau)d\tau\right\}\phi'\,ds=0,
 \end{equation}
for all $\eta,\phi\in\mathcal{H},\,\lambda\in(0,\infty)$. Arguing as in the proof of Proposition \ref{P1}, \eqref{lwko} implies
 \begin{equation}\label{u3sf2}
 	\begin{aligned}
 		&\epsilon \eta'''-\lambda^2\ddot{W}^*\left(\frac{1}{\lambda}\right)\eta'=-k\lambda^5\int_0^s\eta(\tau)d\tau+C \text{ in } (0,1),
 		& \eta''(0)=\eta''(1)=0,
 	\end{aligned}
 \end{equation} 
 where $C$ is a constant. As before, we may rigorously differentiate \eqref{u3sf2}\textsubscript{1} to obtain the linearized problem
 \begin{equation}
 \begin{aligned}
 \label{lin}
 	& \epsilon\eta^{iv}-\lambda^2\ddot{W}\left(\frac{1}{\lambda}\right)\eta''+k\lambda^5\eta=0 \text{ in } (0,1),\\
 	& \eta(0)=\eta(1)=\eta''(0)=\eta''(1)=0,
 \end{aligned}
 \end{equation}
 which is equivalent to \eqref{Ldef}. 
  
 The linear problem \eqref{lin} admits nontrivial solution of the form $\eta_n(s)=\sin(n\pi s)$ whenever $\lambda=\lambda_n$ is a root of the characteristic equation
\begin{equation}\label{char}
\epsilon (n\pi)^4+\lambda^2\ddot{W}^*\left(\frac{1}{\lambda}\right)(n\pi)^2+\lambda^5 k=0.
\end{equation} 
We find such roots by plotting 
\begin{equation}\label{char1}
	k(\lambda)=-\epsilon\frac{(n\pi)^4}{\lambda^5}-\ddot{W}^*\left(\frac{1}{\lambda}\right)\frac{(n\pi)^2}{\lambda^3},
\end{equation} 
for different values of $n$, and looking for intersections with horizontal lines, $k=\text{constant}$. Figure \ref{kvsl} depicts the first five curves for $W^*$ given by \eqref{lp} and parameter values $\epsilon=0.03,\,\beta=3$. We concentrate only on \textit{generic} values of $k$: There are either two or no intersections for a given value of $n$; values where two curves for different values of $n$ intersect are avoided. For example, for $k=2$, there are no roots for $n=1,2$, while there are two roots for each of the other values for $n>2$. For a given generic value of $k$, the smallest possible root of \eqref{char}, denoted $\lambda=\lambda_n$, is the critical (potential) bifurcation. For example, at $k=2$, the smallest root $\lambda_3$ occurs on the $n=3$ curve, while for $k=2.5$, the smallest root $\lambda_4$ occurs on the $n=4$ curve. These are marked at the points P and Q, respectively, on Figure \ref{kvsl}. The graphs give good approximate values, and we find $\lambda_3=2.4490$ and $\lambda_4=2.8561$ via Newton's method. 

   \begin{figure}
 \centering
\includegraphics[width=0.8\linewidth]{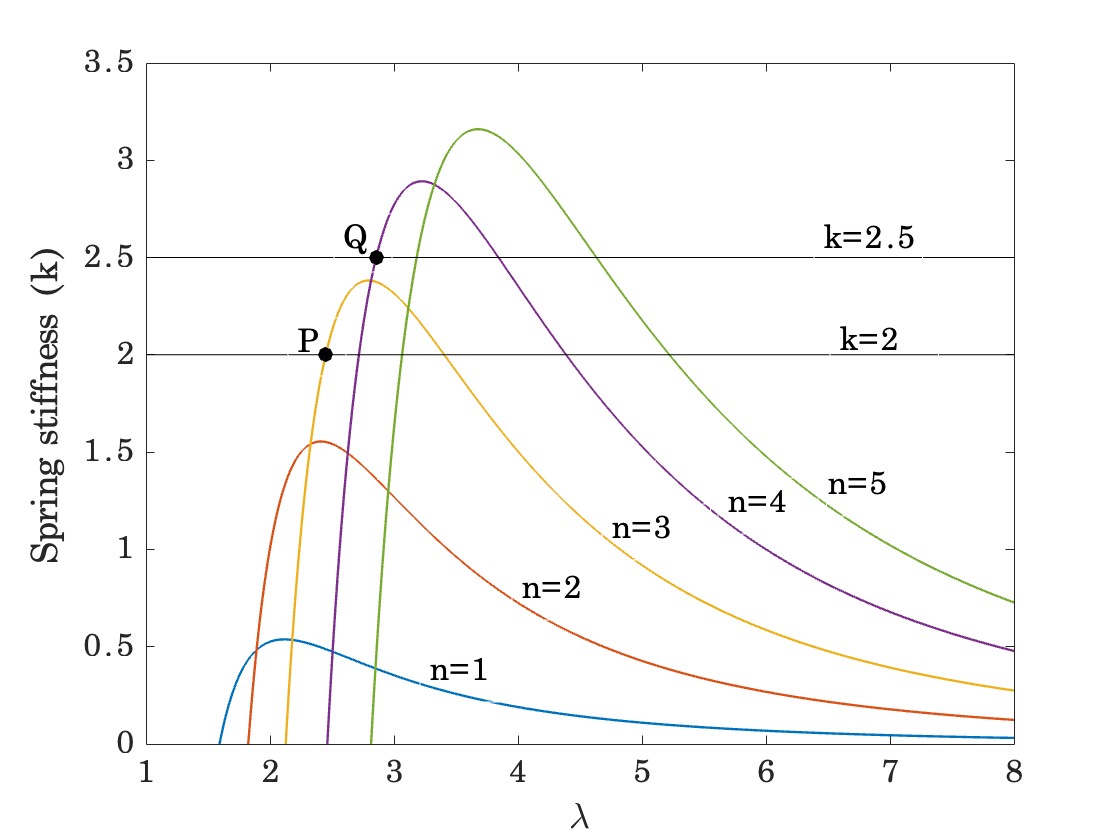}
\caption{First five curves ($n$ equal to 1 to 5) of \eqref{char1} for $W^*$ given by \eqref{lp} and parameter values $\epsilon=0.03,\,\beta=3$. For $k=2$, $n=3$, $\lambda_3=2.449$; for $k=2.5$, $n=4$, $\lambda_4=2.8561$.}   
\label{kvsl}
\end{figure}    

Before exploring bifurcation, we show that there is an exchange of stability along the trivial solution path, $u\equiv 0$, for all $1\leq \lambda <\infty$, at a critical root of the characteristic equation. As such, we first compute the second variation of the functional \eqref{uprob}:
\begin{equation}\label{ustab}
	\delta^2 J^*[\lambda,u]=\int_0^1\left[\epsilon(\eta'')^2+\lambda^2\ddot{W}^*\left(\frac{1+u'}{\lambda}\right)(\eta')^2+\lambda^5k\eta^2\right]\,ds,
\end{equation}
for all $\eta\in\mathcal{H}$. By virtue of Proposition\ \ref{P1}, we work in the class of functions $C^4[0,1]$ satisfying the boundary conditions \eqref{lin}\textsubscript{2}. Integration by parts and subsequent evaluation at $u\equiv 0$ yields
\begin{equation}\label{stab_triv}
	\delta^2 J^*[\lambda,0]=\int_0^1\left[\epsilon\eta^{iv}-\lambda^2\ddot{W}^*\left(\frac{1}{\lambda}\right)\eta''+\lambda^5 k \eta\right]\eta\,ds.
\end{equation}
Generally, if $\delta^2 J^*[\lambda_0,0]>0$ for all $\eta\in\mathcal{H}$, then we say that an equilibrium $u\equiv 0$ at $\lambda=\lambda_0$ is \textit{locally stable}. If $\delta^2 J^*[\lambda_0,0]$ is indefinite, then we say that the equilibrium there is \textit{unstable}. We note that the term in brackets in \eqref{stab_triv} agrees with the left side of equation \eqref{lin}. Now consider the eigenvalue problem
\begin{equation}\label{char2}
	 \begin{aligned}
 	& \epsilon\eta^{iv}-\lambda^2\ddot{W}\left(\frac{1}{\lambda}\right)\eta''+\lambda^5k\eta=\lambda^5\sigma\eta \text{ in } (0,1),\\
 	& \eta(0)=\eta(1)=\eta''(0)=\eta''(1)=0,
 \end{aligned}
 \end{equation}
subject to the same boundary conditions as in \eqref{lin}, where $\lambda^5 \sigma$ is the scaled eigenvalue. Clearly the same nontrivial solutions $\eta_n=\sin(n\pi s)$ arise for roots of the characteristic equation \eqref{char}, but now with $k-\sigma$ in place of $k$. The roots can again be deduced graphically as in Figure \ref{kvsl}; call the smallest root $\tilde{\lambda}_n$. In particular, note that $\sigma>0\Leftrightarrow\tilde{\lambda}_n<\lambda_n$. And $\sigma<0\Leftrightarrow\tilde{\lambda}_n>\lambda_n$ until a second root corresponding to the same $n$ appears. We conclude that \eqref{stab_triv} is positive, and hence, the trivial solution is stable for $\tilde{\lambda}_n<\lambda_n$. For $\tilde{\lambda}_n>\lambda_n$ the solution is unstable, at least until a second root for the same value of $n$ appears.

To establish bifurcation at $(\lambda_n,0)$, it is sufficient to show that $L'(\lambda_n)\eta_n$ does not belong to the range of the linear operator $L(\lambda_n)$, where $L'(\lambda)$ denotes partial differentiation with respect to $\lambda$ \cite{CR}. Since the operator defined by the left side of \eqref{Ldef} is formally self-adjoint, partial differentiation of \eqref{lwko} with respect to $\lambda$ leads to the sufficient condition
\begin{equation*}
\begin{aligned}
	\langle \eta_n, L'(\lambda_n) (\lambda_n,0)\eta_n\rangle=&\left\{n^2\pi^2\left[2\lambda_n\ddot{W}^*\left(\frac{1}{\lambda_n}\right)-\dddot{W}^*\left(\frac{1}{\lambda_n}\right)\right]+5\lambda_n^4k\right\}\\
	&\qquad\qquad\qquad\qquad\qquad\qquad\qquad\cdot\int_0^1\cos^2(n\pi s)\,ds\neq 0,
\end{aligned}
\end{equation*}
or 
\begin{equation}
	\label{CRcond}n^2\pi^2\left[2\lambda_n\ddot{W}^*\left(\frac{1}{\lambda_n}\right)-\dddot{W}^*\left(\frac{1}{\lambda_n}\right)\right]+5\lambda_n^4k\neq0.
\end{equation}
Notice that \eqref{CRcond} simply states that the root $\lambda=\lambda_n$ of \eqref{char} is simple. 

\begin{theorem}
	Let $\mathcal{S}$ denote the closure of the set of all nontrivial solution pairs of \eqref{Fdef}. Assume that $\lambda=\lambda_n$ is simple root of \eqref{char}, viz., condition \eqref{CRcond} is satisfied. Then $\mathcal{S}$ has a component $\Sigma_n\subset\mathbb{R\times\mathcal{K}}$ containing $(\lambda_n,0)$ satisfying at least one of the following:
	\begin{enumerate}[(i)]
		\item\label{gf1}$\Sigma_n$ is unbounded in $\mathbb{R}\times\mathcal{H}$;
		\item\label{gf2}$(\lambda_*,0)\in\Sigma_n$, where $\lambda_*\neq\lambda_n$.
	\end{enumerate}
	Moreover, in a sufficiently small neighborhood of $(\lambda_n,0)$, $\Sigma_n$ is characterized uniquely by a (pitchfork) curve of solutions of the form 
	\begin{equation}\label{solut}
	\lambda=\lambda_n+o(\tau),\, u=\tau\sin(n\pi s)+o(\tau),\, \text{as}\, \tau\rightarrow 0.
\end{equation}  
\end{theorem}
\begin{proof}
	As already discussed, the presumed simplicity of the root \eqref{CRcond} insures the existence of the unique local path of non-trivial solutions of \eqref{Fdef} and hence of \eqref{ubc2}. That the path is a so-called pitchfork, viz., $o(\tau)$ vs.\ $O(\tau)$ in \eqref{solut}, follows either by direct calculations or via ``hidden'' symmetry arguments, e.g., \cite{HP}.
	
	To verify the alternatives (\ref{gf1}), (\ref{gf2}), we first note that $L(\lambda):\mathcal{H}\rightarrow\mathcal{H}$ is a compact map. Accordingly, the Leray-Schauder linearization principle shows that the topological degree of $u\mapsto\epsilon u+F(\lambda,u)$ on some ball $B_r(0)\in\mathcal{H}$, centered at $u\equiv 0$ of sufficiently small radius $r>0$, is given by
	\begin{equation}
		\deg(\epsilon\mathcal{I}+F(\lambda,\cdot),B_r(0),0)=(-1)^{m(\lambda)},
	\end{equation}  
	where $m(\lambda)$ denotes the number of negative eigenvalues counted by algebraic multiplicity of the linear map $\epsilon\mathcal{I}+L(\lambda)$, for $\lambda\in(a,\lambda_n)\cup(\lambda_n,b)$, where both $\lambda_n-a>0$ and $b-\lambda_n>0$ are sufficiently small. We observe that \eqref{char2} is the same as 
	\begin{equation}
		\epsilon\eta+L(\lambda)\eta=\lambda^5\sigma T\eta, \text{ for all }\eta\in\mathcal{H},
	\end{equation}
	where $T\eta:=\int_0^s\eta(\tau)d\tau$. Consequently, the exchange-of-stability argument given just after \eqref{char2} imply
	\begin{equation}
		\deg(\epsilon\mathcal{I}+F(\lambda_1,\cdot),B_r(0),0)\neq\deg(\epsilon\mathcal{I}+F(\lambda_2,\cdot),B_r(0),0),
	\end{equation}
	for $\lambda_1\in(a,\lambda_n)$ and $\lambda_2\in(\lambda_n,b)$. Thus, (\ref{gf1}), (\ref{gf2}) follow from the theorem of Rabinowitz \cite{R}.
\end{proof}

\section{Numerical Continuation}\label{nummeth}
The existence results of the previous section do not address if cracks actually emerge on global solution branches. For that, we turn to numerical continuation methods using a finite element discretization. Consider the weak form of \eqref{u3ineq}
\begin{equation}\label{wf}
	\int_0^1 \left\{\epsilon u''\psi''+\lambda^3\dot{W}^*\left(\frac{1+u'}{\lambda}\right)\psi'+\lambda^5ku\psi\right\}ds=\int_{\mathcal{B}_u} \{\mu\psi' +\zeta (u'+1)\}\,ds.
\end{equation}
where the measure $\mu$ in \eqref{u3ineq} is now assumed to belong to $L^\infty(0,1)$, and $\zeta\in L^\infty(0,1)$ is defined as an admissible variation.  The function $\mu$ is interpreted as the Lagrange multiplier field associated with the inequality constraint $u'+1\geq 0$ such that 
\begin{equation}
	\mu(s)\geq 0, \text{ and } \mu(s)(u'(s)+1)=0 \text{ for all } s\in[0,1].
\end{equation}
 
The domain $[0,1]$ is divided into $N$ elements of equal length and $N+1$ nodes. The values of $u$, $u'$,  and $\mu$ represented by $u_k$, $u'_k$ and $\mu_k$ at the $k^{th}$-node are unknowns. To evaluate the left side of \eqref{wf}, we interpolate $\{u_k\}$ using piecewise cubic Hermite shape functions. As used in Euler-Bernoulli beam problems, these provide sufficient continuity across elements such that interpolated functions lie in ${H}^2(0,1)$ \cite{TH}. The integrand is evaluated at Gauss points and the complete integral is evaluated using quadrature. 

The right side of \eqref{wf} is evaluated as follows:
\begin{equation}\label{sright}
	\int_{\mathcal{B}_u} \{\mu\psi' +\zeta (u'+1)\}\,ds\approx\sum_{s_k\in \mathcal{B}_u^h}\{\mu_k\psi'_k+\zeta_k(u'_k+1)\}.
\end{equation}
where $\mathcal{B}_u^h:=\{s_k: k\in 1,2,\cdots, N+1;\,s_k\in \mathcal{B}_u\}$ is the finite dimensional approximation of $\mathcal{B}_u$. The interpolation \eqref{sright} implies that constraints $u'_k+1= 0$ are imposed on the discretized problem at all nodal points $s_k\in \mathcal{B}_u^h$. For a solution $(\lambda,u)\in (0,\infty)\times\mathcal{K}^o$ in the interior of the cone, the set $\mathcal{B}_u^h$ is empty and the right side of \eqref{wf} is zero. When $\mathcal{B}_u^h$ is not empty, it contains position of nodes where $u'_k+1=0$. This is determined using the active-set method \cite{NW} and is discussed further in the next subsection.    

The active-set method works within pseudo-arc length continuation \cite{Ke}. We drop the natural parametrization of the solution by $\lambda$ and increment over the arc-length of the curve. This allows the algorithm to traverse past singularities on the solution curve. 

\subsection{Active-set method}
The active-set $\mathcal{B}^h_u$ is the set of all nodes where the inequality constraint is active, i.e. $u'_k+1=0,$ for $s_k\in\mathcal{B}_u^h$. If $\mathcal{B}^h_u$ is known, then the inequality constraints are imposed as equality constraints at the nodal points $s_k\in\mathcal{B}^h_u$, and the solution is found using Newton's method. However, the active-set $\mathcal{B}^h_u$ is generally not known and is found in conjunction with equilibrium. We start by making a guess of the optimal $\mathcal{B}_u^h$ and solving for equilibrium. If the solution satisfies conditions
\begin{equation}\label{cond1}
	\mu_k\geq 0 \text{ for all } s_k\in\mathcal{B}^h_u,
\end{equation}
and 
\begin{equation}\label{cond2}
	 u'_k+1\geq 0 \text{ for all } s_k\notin \mathcal{B}_u^h,
\end{equation}
then the solution is in the feasible region. If the conditions fail, then one node is either added in or removed from the active-set $\mathcal{B}^h_u$. If \eqref{cond1} fails, then the node with minimum value of $\mu$ is removed from the active-set. If \eqref{cond1} is satisfied but \eqref{cond2} fails, then the node with minimum value of $u'$ is added in the active set. In each iterate only one node is moved, and the procedure is repeated until a solution in the feasible region is found. More details related to convergence of the method are in \cite{NW}.

The procedure is sensitive to the initial guess of the active-set $\mathcal{B}^h_u$ and the initial guesses of $\{u_k\},\,\{u'_k\},$ and $\{\mu_k\}$ in Newton's methods. Different choices leads to different iteration sequences and often different equilibrium solutions. Since the algorithm works within a continuation scheme, we use information from previous point as initial guesses while solving for a new point on the solution curve.

\subsection{Stability analysis}
A solution $(\lambda_0,u_0)\in \mathcal{K}$ is \emph{locally stable} if $\delta^2J^*[\lambda_0,u_0]$ defined in \eqref{ustab} is positive for all $\eta\in\mathcal{H}$, $\eta'= 0$ on $\mathcal{B}_u$. Although variations $\eta'>0$ on $\mathcal{B}_u$ lie in the admissible set of the minimization problem, they imply healing in the material. This phenomena is not considered here, and such variations are not included in the determination of stability of an equilibrium. 

We determine stability of an equilibrium from eigenvalues of the stiffness matrix obtained using the finite element discretization. This stiffness matrix is square and symmetric of size $2N+M$, where $M$ is number of nodes in $\mathcal{B}^h_u$. It is of the form
\begin{equation}
	K=\begin{bmatrix}
		G & A^T\\
		A & 0
	\end{bmatrix}
\end{equation}
where $G$ is a square matrix of size $2N$ obtained from taking variation of left side of \eqref{wf} and imposing geometric boundary conditions. Matrix $A$ is of size $M\times (2N)$ obtained from the derivative of constraints $u'_k+1=0$ on $\mathcal{B}^h_u$. We define the inertia of the matrix $K$ as
\begin{equation}
	\inertia(K)=(n_+,n_-,n_0)
\end{equation}
where $n_+$, $n_-$, and $n_0$ are the number of positive, negative, and zero eigenvalues of $K$. 

Upon discretization, the space of variations $\eta\in\mathcal{H}$, $\eta'=0$ on $\mathcal{B}_u$ is approximated using finite dimensional vectors $\{\eta_k\}$ and $\{\eta_k'\}$ such that $\eta'_k=0$ for all $s_k\in\mathcal{B}^h_u$. These vectors lie in the null space of matrix $A$. Let $Z$ be a $(2N)\times M$ sized matrix such that its columns span the null space of $A$. Then the equilibrium solution is stable if the matrix $Z^TGZ$ is positive definite. It follows from Theorem 16.3 in \cite{NW} that $Z^TGZ$ is positive definite if 
\begin{equation}\label{inertia_cond}
	\inertia(K)=(2N,M,0).
\end{equation}  
For every equilibrium point we evaluate the inertia of $K$. If it matches \eqref{inertia_cond}, the solution is locally stable. Otherwise, it is unstable.

\section{Numerical results}\label{numres}

We present numerical results for the data sets discussed in Section \ref{globalbif}. As before, we set $\epsilon=0.03$ and $\beta=3$, for which we consider the cases $k=2$ and $k=2.5$. Recall that for $k=2,$ the first bifurcation is at $\lambda=2.4490$, corresponding to point P in Figure\ \ref{kvsl}, with associated mode number $n=3$. For $k=2.5$, the first bifurcation is at $\lambda=2.8561$, corresponding to point Q in Figure\ \ref{kvsl}, and the associated mode number is $n=4$. We use 100 elements for the finite element discretization and relative and absolute tolerances of $10^{-7}$ in the Newton's method. Further refinement of the mesh did not provide more accuracy in the results presented here.

In Figure\ \ref{stress3}, both the macroscopic stress and energy for the brittle bar vs. the stretch parameter $\lambda$ for the case $k=2$ are shown. At any solution pair $(\lambda,u)$, we compute the energy $I^*$ using \eqref{uprob}, and we compute the macroscopic stress via
\begin{equation}\label{sigdef}
	\sigma(\lambda)=\left.\frac{d }{d\tau}I^*[\tau,u(\tau)]\right\rvert_{\tau=\lambda},
\end{equation}
which follows from Castigliano's first theorem \cite{TY}. For instance, along the trivial solution $u\equiv 0$, $I^*[\lambda,0]=\lambda W^*(1/\lambda)=W(\lambda)=\frac{\beta}{6}\left(1-\frac{1}{\lambda}\right)^2$, the latter two equalities of which follow from the definition of $W^*$ and \eqref{lp}. From \eqref{sigdef}, the macroscopic stress is then given by $\sigma=\dot{W}(\lambda)$, which is simply the uniform-stretch constitutive law for the brittle rod. For numerical solutions, we compute $I^*$ at points along the solution curve and then evaluate \eqref{sigdef} via linear interpolation. In any case, \eqref{sigdef} represents the difference between the stress required to pull the composite rod minus the stress in the pseudo-rigid elastic bar at $\lambda$.

\begin{figure}[H]
\centering
\includegraphics[width=\linewidth]{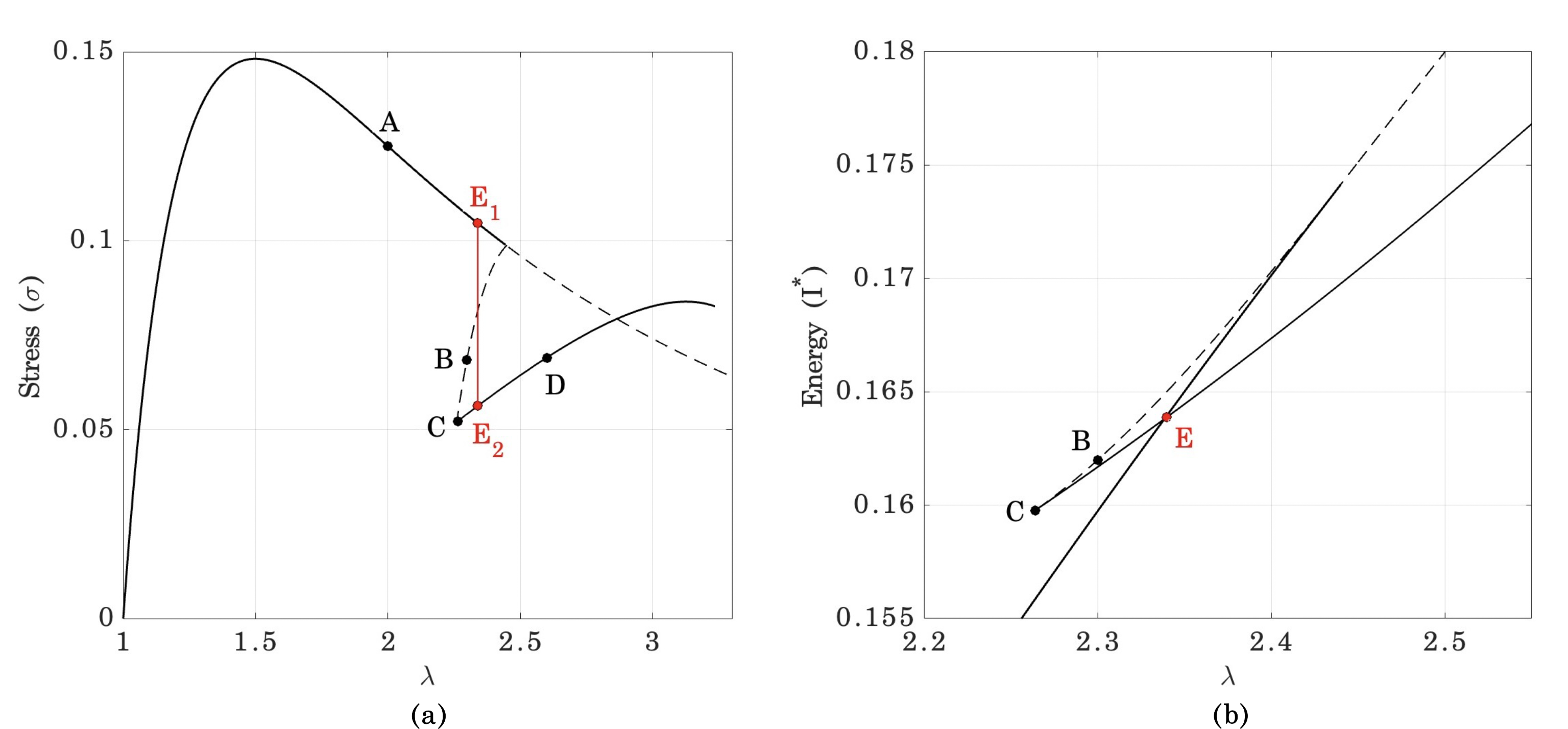}
\caption{Stress and energy curves for parameter values $\epsilon=0.03,\,\beta=3,\, \,k=2$. (a) Stress ($\sigma$) vs stretch ($\lambda$). (b) Energy ($I^*$) vs stretch ($\lambda$). Solid and dashed black lines show stable and unstable equilibria respectively. Red line shows transition between the two stable solutions.}  
\label{stress3}
\end{figure}

The solid and dashed lines shown in Figure\ \ref{stress3} represent stable and unstable equilibria, respectively. The points A, B, C, and D shown are chosen to illustrate various features of solutions. The deformation $f$ on $[0,1]$ and the inverse deformation gradient $H$ on $[0,\lambda]$ are depicted in Figure\ \ref{disp2} for each of those points. The marker E in red corresponds to the value of $\lambda$ at which the energy of the two stable solutions are equal.

The trivial solution $u\equiv 0$ is locally stable for all $\lambda$ values less than the bifurcation point at $\lambda=2.4490$, after which stability is lost. These are homogeneous solutions, and the configuration corresponding to point A in Figure\ \ref{stress3} is depicted in Figure\ \ref{disp2}(a); the constant inverse strain is shown in Figure\ \ref{disp2}(b). The bifurcation is a subcritical pitchfork; in particular, it is locally unstable. Globally, the solutions along the two ``sides'' of the pitchfork are equivalent - related by reflection symmetry. As such, their projections overlap in both Figure\ \ref{stress3}(a) and (b). The configuration corresponding to point B is shown in Figure\ \ref{disp2}(c), and the associated inverse strain in Figure\ \ref{disp2}(d).
\begin{figure}
\centering
\includegraphics[width=\linewidth]{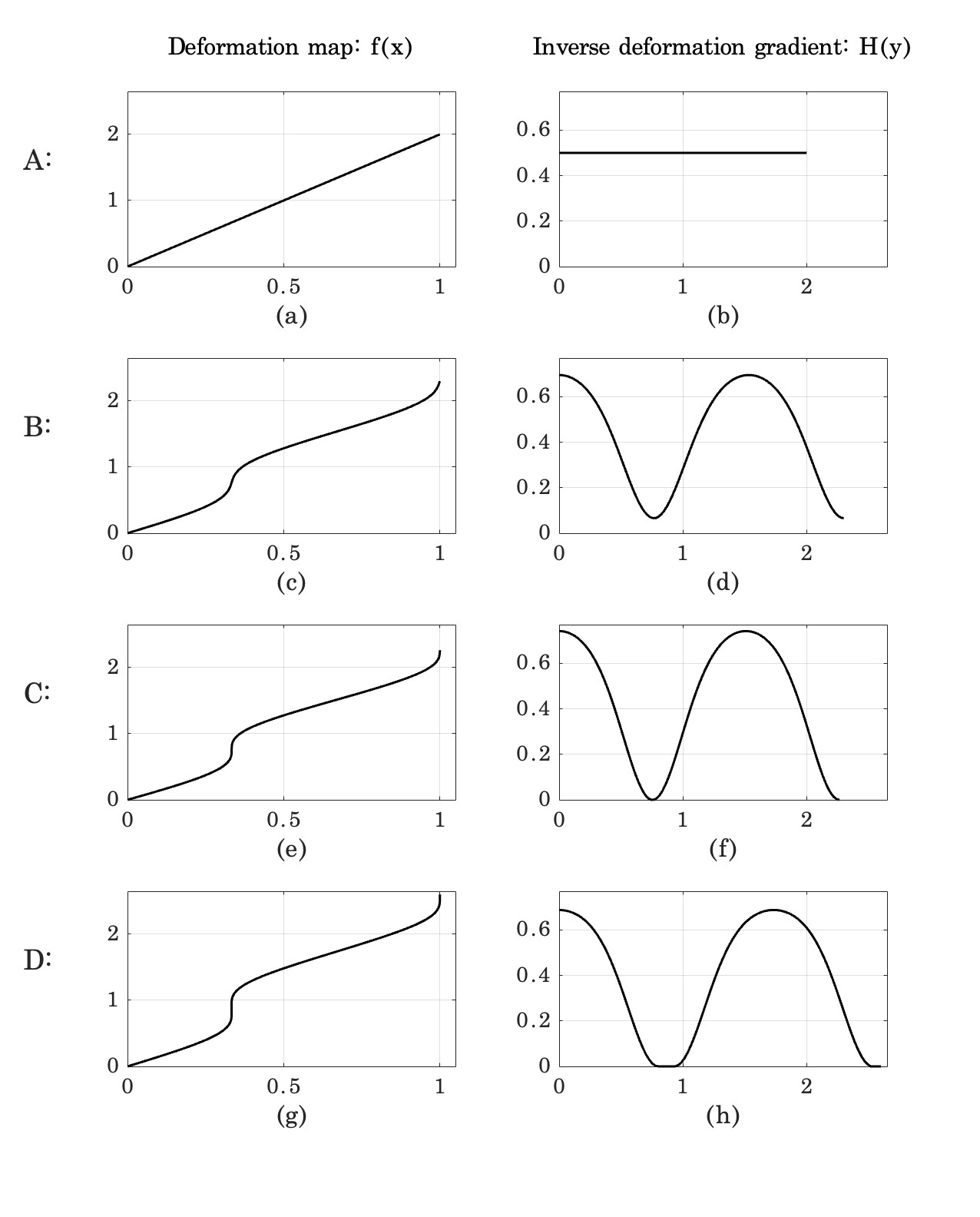}
\caption{Deformation $f(x)$ and inverse deformation gradient $H(y)$ at 4 different points for comparison (see Figure\ \ref{stress3}). Point A is before bifurcation. Point B is on the unstable non-trivial post-bifurcation curve. Point C is the point where cracks first appear. Point D is on the stable curve post-fracture.}  
\label{disp2}
\end{figure}

Point C in Figure\ \ref{stress3} corresponds to the nucleation of cracks. The configuration and inverse strain with two zero values are shown in Figure\ \ref{disp2}(e) and (f), respectively. Since the cracked outer layer continues to interact with the pseudo-rigid bar, the stress does not drop to zero (as in \cite{RH}). As shown in Figure \ref{stress3}, a stable branch emerges from point C, along which the cracks widen and the outer layer continues to interact with the pseudo-rigid foundation. Figure \ref{disp2}(g) and (h) depict the configuration and inverse strain fields, respectively, corresponding to the point D in Figure \ref{stress3}, clearly demonstrating the growth of the cracks at $x=1/3$ and $x=1$ in the reference configuration. We remark that two other distinct solution branches emanate from the point D as well. However, both branches contain ``healing'' solutions, viz., one of the two cracks closes up. Also, both branches are unstable. As such, we do not consider them here. Finally, the cracked solutions in Figure \ref{stress3}, along the branch containing the segment C-D, have lower energy than the trivial, homogeneous solutions, at all values of $\lambda>2.3385$, which corresponds to the point E.  

Figures\ \ref{stress4}-\ref{disp4} present the same information just discussed, but for the case $k=2.5$. The difference here with $n=4$ is that the two sides of the pitchfork lead to ostensibly different fracture patterns. In particular, Figure\ \ref{disp3}(e)-(h) corresponding to points H and I in Figure \ref{stress4}, indicate three cracks - one at each end and one at the midpoint. On the other side of the pitchfork with typical solutions shown in Figure\ \ref{disp4}(e)-(h) and also corresponding to points H and I in Figure\ \ref{stress4}, we observe 2 symmetrically spaced interior cracks. In spite of these observations, the solutions and their energies again overlap on Figure\ \ref{stress4} - just as in the previous case $k=2$ $(n=3)$. As discussed in \cite{RH}, the strain gradient surface energy associated with fracture implies that fewer cracks ``cost'' less. But end cracks are different from interior cracks: the former involve only one fractured ``face'', whereas the latter involves two. In other words, the two end cracks in Figure\ \ref{disp3}(e)-(h) are equivalent to an interior crack as in Figure\ \ref{disp4}(e)-(h). This explains the overlap in Figure\ \ref{stress4}(a) and (b). 

\begin{figure}[H]
\centering
\includegraphics[width=\linewidth]{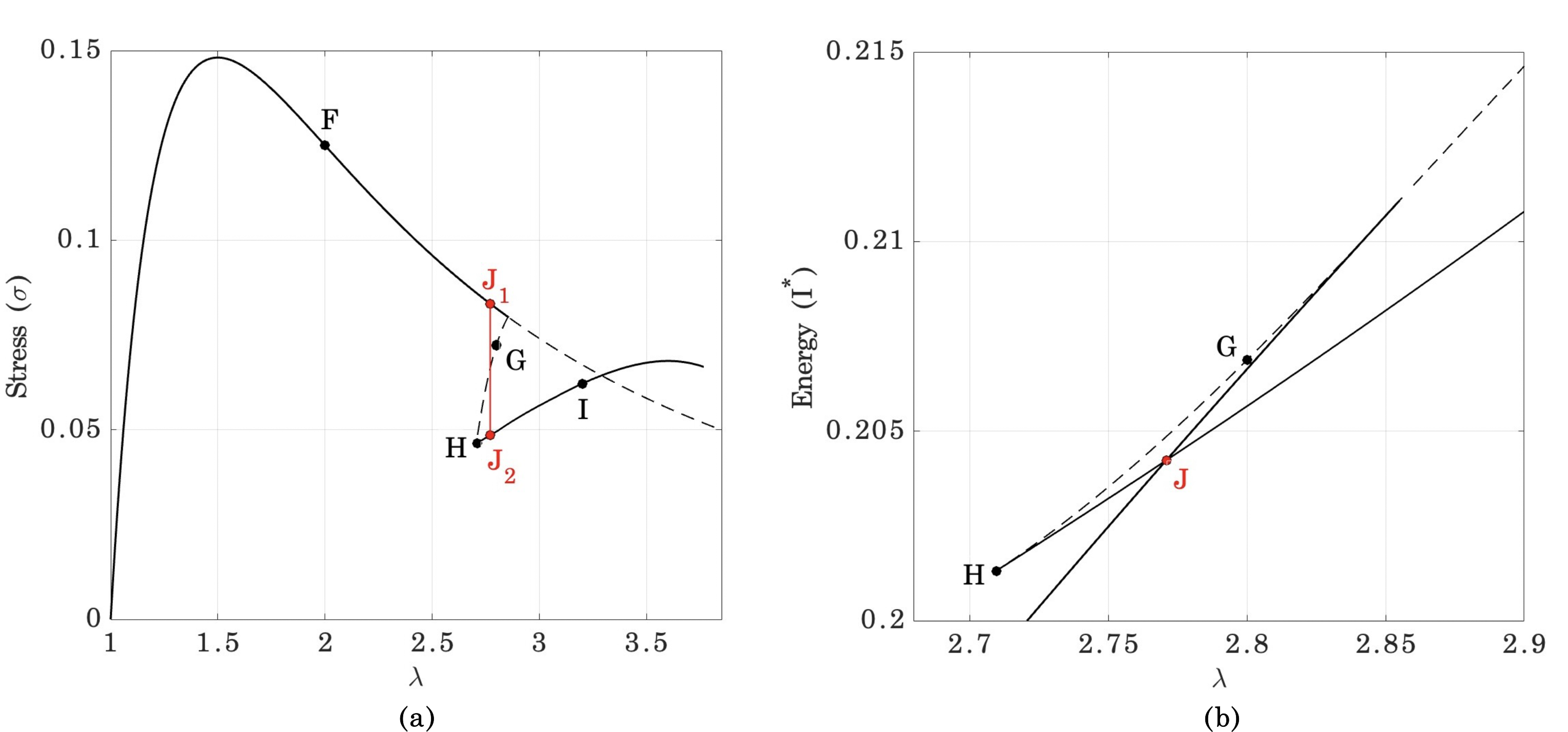}
\caption{Stress and energy curves for parameter values $\epsilon=0.03,\,\beta=3,\, \,k=2.5$. (a) Stress ($\sigma$) vs stretch ($\lambda$). (b) Energy ($I^*$) vs stretch ($\lambda$). Solid and dashed black lines show stable and unstable equilibria respectively. Red line shows transition between the two stable solutions.}  
\label{stress4}
\end{figure}
 
\begin{figure}[H]
\centering
\includegraphics[width=\linewidth]{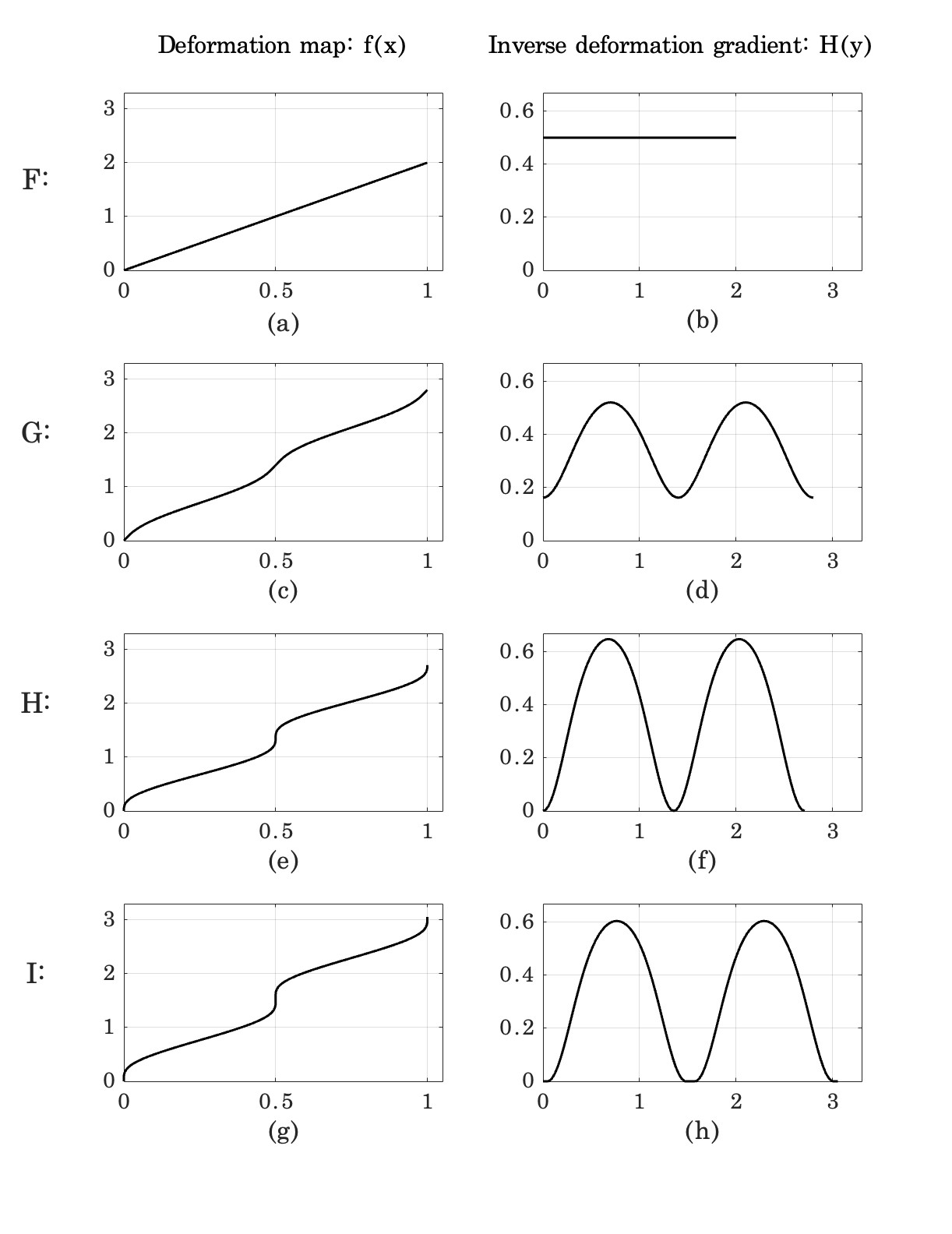}
\caption{Deformation $f(x)$ and inverse deformation gradient $H(y)$ at 4 different points for comparison (see Figure\ \ref{stress4}). Point F is before bifurcation. Point G is on the unstable non-trivial post-bifurcation curve. Point H is the point where cracks first appear. Point I is on the stable curve post-fracture.}  
\label{disp3}
\end{figure}

\begin{figure}[H]
\centering
\includegraphics[width=\linewidth]{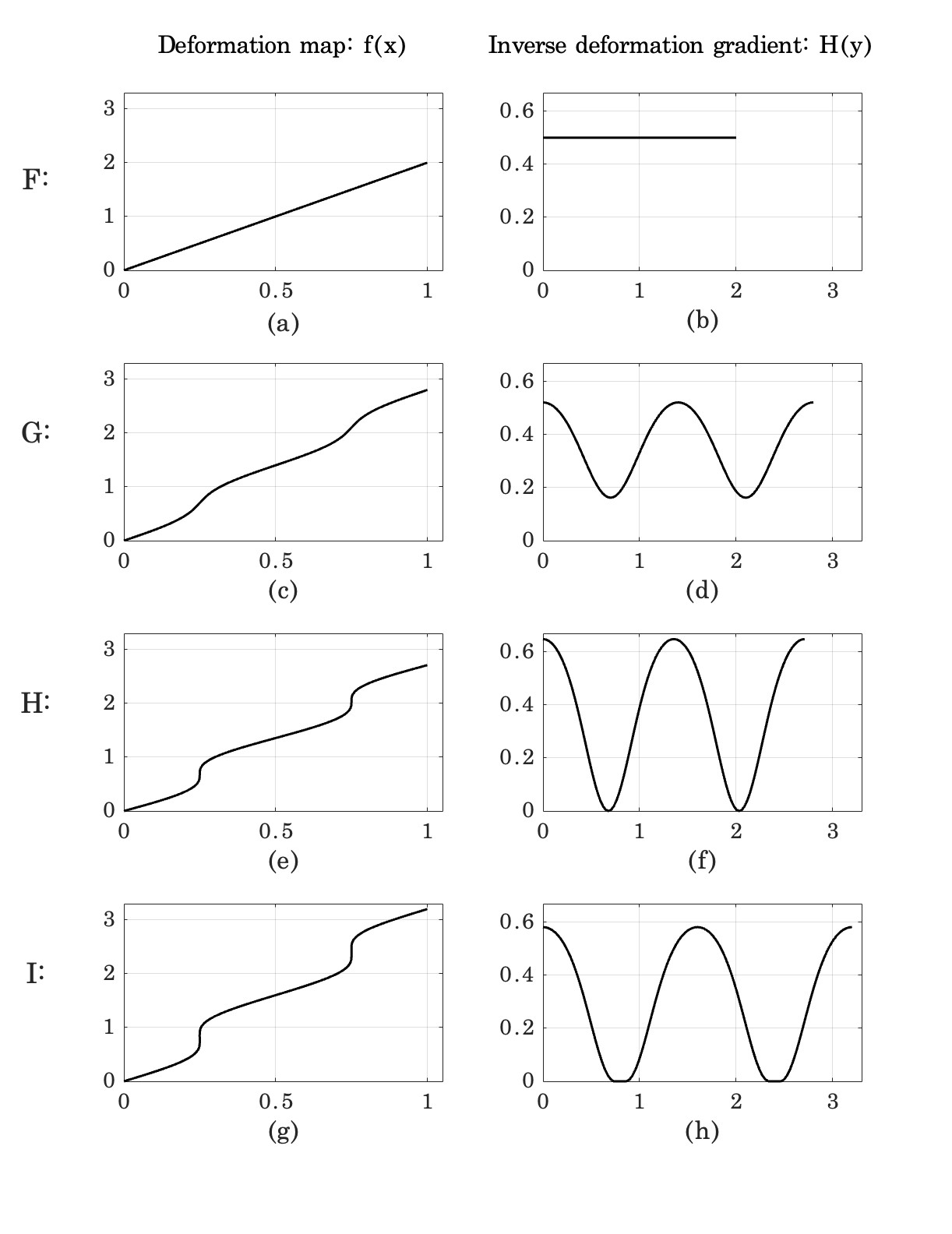}
\caption{Deformation $f(x)$ and inverse deformation gradient $H(y)$ at 4 different points for comparison (see Figure\ \ref{stress4}). Point F is before bifurcation. Point G is on the unstable non-trivial post-bifurcation curve. Point H is the point where cracks first appear. Point I is on the stable curve post-fracture.}  
\label{disp4}
\end{figure}

\section{Concluding remarks}\label{conc}
We present solutions exhibiting the spontaneous nucleation of cracks and their development in a simple model for thin composite fibers. We employ the inverse-deformation approach of \cite{RH}, incorporating neither a damage field nor pre-existing cracks. For specific parameter sets, critical patterns comprising two or three fractures in the outer layer of the composite are illustrated. It should be clear, that we can choose other parameter sets leading to as many cracks in a pattern as desired.

We emphasize the new difficulty here of addressing the unilateral constraint. This is not an issue in \cite{RH}, due to the fact that effective stress in that problem drops to zero (identically) upon nucleation, after which cracks open freely under hard loading. In contrast, the fractured outer layer continues to interact with the elastic core in our problem. We compute these solutions via the active-set method \cite{NW} as applied to our discretized system. 

If we drop the assumption leading to \eqref{assump}, i.e., make no assumption leading to a pseudo-rigid foundation, the problem is more difficult. For instance, \eqref{ELeq1} becomes a field equation; a two-rod model is now required. Moreover, from the experimental results motivating our work, it is reported that the elastic core eventually fractures as well \cite{NC}. To predict that phenomenon, a stronger, brittle model for the core is also needed. This leads to an even more challenging problem, with each rod now requiring its own inverse-deformation description. These problems are stepping stones to the more realistic case of a 2D brittle solid, all of which we plan to pursue in the future. 

\vspace{1em}
\textbf{Acknowledgements} This work was supported in part by the National Science Foundation through grant DMS-2006586, which is gratefully acknowledged. We thank Phoebus Rosakis, Patrick Farrell, David Bindel, Maximilian Ruth, and Gokul Nair for useful conversations.

% BibTeX users please use one of
\bibliographystyle{hacm.bst}      % basic style, author-year citations
%\bibliographystyle{spmpsci}      % mathematics and physical sciences
%\bibliographystyle{spphys}       % APS-like style for physics
%\bibliography{}   % name your BibTeX data base
\bibliography{pseudo.bib}
\end{document}